\documentclass{amsart}

\usepackage{amssymb}
\usepackage[pagebackref]{hyperref}
\usepackage{cite}
\usepackage[all]{xy}
\usepackage{enumerate}
\usepackage{mathrsfs}
\usepackage{color}
\usepackage{tikz}
\usepackage{microtype} 
\usepackage{geometry}
\newtheorem{theorem}{Theorem}[section]
\newtheorem{lemma}[theorem]{Lemma}
\newtheorem{corollary}[theorem]{Corollary}
\newtheorem{proposition}[theorem]{Proposition}

\theoremstyle{definition}

\theoremstyle{plain}
\newtheorem{remark}{Remark}[section]
\newtheorem*{problem}{Problem}

\numberwithin{equation}{section}

\DeclareMathOperator{\divg}{div}
\DeclareMathOperator{\supp}{supp}

\DeclareMathOperator{\Ric}{Ric}

\DeclareMathOperator{\Vol}{Vol}

\newcommand{\bbR}{\mathbb{R}}
\newcommand{\bbC}{\mathbb{C}}
\newcommand{\bbZ}{\mathbb{Z}}

\newcommand{\iu}{\sqrt{-1}}

\newcommand{\dd}{\mathop{}\!\mathrm{d}}

\renewcommand{\Re}{\operatorname{Re}}

\allowdisplaybreaks[4]

\begin{document}

\title[\(L^p\) Liouville theorems on gradient K\"ahler-Ricci solitons]{\(L^p\) Liouville theorems for pluriharmonic functions on gradient K\"ahler-Ricci solitons}


\author{Guangwen Zhao}
\address{School of Mathematics and Statistics, Wuhan University of Technology, Wuhan 430070, China}
\curraddr{}
\email{gwzhao@whut.edu.cn}
\thanks{This work is partially supported by the National Natural Science Foundation of China (12001410)}





\dedicatory{}

\begin{abstract}
    We study Liouville-type theorems for real-valued pluriharmonic functions on complete gradient K\"ahler-Ricci solitons under gradient integrability assumptions. For a complete gradient K\"ahler-Ricci soliton \((M,g,J,f)\) and a real-valued pluriharmonic function \(u\), we investigate conditions under which \(u\) must be constant. By introducing a globally defined holomorphic quantity induced by the soliton potential, we obtain new Liouville-type results beyond the range available for harmonic functions. In the steady case, we prove that \(u\) is constant whenever
    \[
        \int_M|\nabla u|^p\dd v<\infty
    \]
    for some \(0<p<\infty\). In the shrinking case, we prove the same conclusion for \(0<p\le 2\). Finally, we construct a complete K\"ahler example showing that the extension to the range \(0<p<1\) relies essentially on the soliton structure and does not hold on general complete K\"ahler manifolds.
\end{abstract}

\maketitle

\tableofcontents

\section{Introduction}

The classic Liouville theorem states that any bounded holomorphic function defined on the complex plane \(\bbC\) must be a constant. Its various extensions are important topics in differential geometry and complex analysis. Among them, obtaining the Liouville theorem through integrability conditions is an important approach. Regarding this, the two important theorems established by Yau \cite{MR417452} in 1976 can be described as follows: He proved that \(L^p\ (p>1)\) harmonic functions on complete Riemannian manifolds must be constants, and also proved that \(L^p\ (p>0)\) holomorphic functions on complete K\"ahler manifolds must be constants. The latter was further extended by Li--Zhang--Zhang \cite{MR3951749} to a class of complete Gauduchon manifolds in 2019. The difference between these two exponent ranges is essential in the present problem.

In this paper we study Liouville-type theorems on complete gradient K\"ahler–Ricci solitons. A Riemannian manifold \((M,g)\) is called a gradient Ricci soliton if there exist a smooth function \(f\), called the potential function, and a constant \(\lambda\) such that
\[
    \Ric+\nabla^2f=\lambda g.
\]
If, in addition, \((M,g,J)\) is K\"ahler, then it is called a gradient K\"ahler–Ricci soliton. In this case \(\nabla f\) is a real holomorphic vector field, or equivalently, the Lie derivative
\[
    \mathcal L_{\nabla f}J=0.
\]
This is also equivalent to \((\nabla f)^{1,0}\) being holomorphic. The interaction between the soliton structure and the underlying complex geometry makes this a natural setting for Liouville-type problems. The soliton is called shrinking, steady, or expanding according as \(\lambda>0\), \(\lambda=0\), or \(\lambda<0\). A smooth function \(u\) on a complex manifold \((M,J)\) is called pluriharmonic if it satisfies \(\partial \bar\partial u=0\). Furthermore, if the complex manifold is endowed with a K\"ahler metric \(g\) (i.e., on the K\"ahler manifold \((M,g,J)\)), pluriharmonicity can also be characterized by
\[
	\nabla^2u(X,Y)+\nabla^2u(JX,JY)=0,\qquad \forall X,Y\in \Gamma(TM),
\]
where \(\nabla \) is the Levi--Civita connection of \(g\). It is well known that, on a K\"ahler manifold, every holomorphic function is complex-valued pluriharmonic, and every pluriharmonic function is harmonic.

The transition from holomorphic functions to real-valued pluriharmonic functions is natural at the level of first derivatives. Indeed, a real-valued function \(u\) is pluriharmonic precisely when it is locally the real part of a holomorphic function, equivalently, \(\partial u\) is a holomorphic \((1,0)\)-form. Thus the condition
\[
    \int_M |\nabla u|^p\,\dd v<\infty
\]
may be viewed as an \(L^p\) condition on a holomorphic differential. There is, however, an important global distinction: a pluriharmonic function need not be the real part of a globally defined holomorphic function. On a gradient K\"ahler–Ricci soliton \((M,g,J,f)\), the soliton potential provides a natural way to overcome this obstruction by producing the globally defined holomorphic function
\[
    H=\left\langle (\nabla f)^{1,0},(\nabla u)^{0,1}\right\rangle =\frac{1}{2}\left(\langle \nabla f,\nabla u\rangle +\iu \langle \nabla f,J\nabla u\rangle \right).
\]
This suggests that Yau’s all-exponent holomorphic Liouville phenomenon may still lead to rigidity for real-valued pluriharmonic functions with finite \(p\)-energy. We are therefore led to the following problem.

\begin{problem}
	Let \((M,g,J,f)\) be a connected complete gradient steady or shrinking K\"ahler–Ricci soliton, and let \(u\) be a real-valued pluriharmonic function. If
    \[
        \int_M|\nabla u|^p\,\dd v<\infty
    \]
    for some \(0<p<\infty\), must \(u\) be constant?
\end{problem}

A number of Liouville-type results have been established for gradient Ricci solitons. Ge--Zhang \cite{MR3759351} proved that every positive \(f\)-harmonic function on a complete gradient shrinking Ricci soliton is constant and established related weighted \(L^p\) Liouville theorems for \(f\)-subharmonic and \(f\)-superharmonic functions. Passing from the drift Laplacian to the ordinary Laplacian is substantially more delicate. Under the additional assumption of constant scalar curvature, Mai--Ou \cite{MR4739247} subsequently proved that every bounded harmonic function on a complete noncompact gradient shrinking Ricci soliton is constant, they also showed that the space of harmonic functions of any fixed polynomial growth order is finite-dimensional. 

Stronger conclusions become available in the K\"ahler setting, where the complex structure can be combined with the soliton equation. Munteanu--Sesum \cite{MR3023848} proved that every harmonic function with finite Dirichlet energy on a complete gradient shrinking K\"ahler–Ricci soliton is constant, their corresponding result for steady gradient Ricci solitons does not require the K\"ahler assumption. Munteanu--Wang \cite{MR3217651} later showed that every bounded holomorphic function on a complete shrinking K\"ahler–Ricci soliton is constant and obtained finite-dimensionality results for holomorphic functions of prescribed polynomial growth. More recently, using the auxiliary real-valued harmonic function \(F=\langle\nabla f,\nabla u\rangle \), Luo \cite{MR4935611} replaced boundedness by gradient integrability and proved that, on a complete steady or shrinking gradient K\"ahler–Ricci soliton, every real-valued pluriharmonic function \(u\) satisfying
\[
    \int_M|\nabla u|^p\dd v<\infty ,\qquad 1<p\le 2,
\]
must be constant.

Compared with the approach of Luo \cite{MR4935611}, which relies on the real-valued harmonic quantity \(F\), our argument is based on the holomorphic quantity \(H\). The holomorphic structure of \(H\) is crucial for passing below the threshold \(p=1\). From the preceding discussion, we see that
\[
   F=2\Re H.
\]
Hence, we may take \(H\) as an auxiliary function in our treatment of Liouville theorem. We partially answer the problem we posed above.

In the steady case, we give a complete affirmative answer, even under a weaker integrability condition.

\begin{theorem}\label{thm-a}
	Let \((M,g,J,f)\) be a connected complete gradient steady K\"ahler-Ricci soliton and let \(u\) be a pluriharmonic function on \(M\). Suppose that \(\Psi :[0,\infty )\to [1,\infty )\) is locally Lipschitz, nondecreasing, and satisfies 
    \begin{equation}\label{eq-z}
    	\int_0^\infty \frac{\dd t}{\Psi (t)}=\infty .
    \end{equation}
    If 
	\begin{equation}\label{eq-y}
		\int_M\frac{|\nabla u|^p}{\Psi (d(o,x))}\dd v<\infty ,
	\end{equation}
	for some \(0<p<\infty \), where \(o\) is a fixed point on \(M\) and \(d(o,x)\) is the distance function from \(o\), then \(u\) is a constant function.
\end{theorem}

Taking \(\Psi \equiv 1\), we obtain the following corollary, which provides a completely affirmative answer to the problem we posed above in the steady case.

\begin{corollary}\label{cor-a}
	Let \((M,g,J,f)\) be a connected complete gradient steady K\"ahler-Ricci soliton and let \(u\) be a pluriharmonic function on \(M\). If
	\[
	    \int_M|\nabla u|^p\dd v<\infty 
	\]
	for some \(0<p<\infty \), then \(u\) is constant.
\end{corollary}

The admissible weights in Theorem~\ref{thm-a} include \(\Psi(t)=1+t,\ \Psi(t)=(1+t)\log (e+t)\), and the usual finite products with iterated logarithms whenever the reciprocal integral in \eqref{eq-z} diverges. Thus \eqref{eq-y} can be strictly weaker than ordinary \(L^p\) integrability. The divergence condition \eqref{eq-z} is sharp in some sense. Specially, we have

\begin{remark}
	Consider the flat complex cylinder \(M=\bbC/(2\pi \iu \bbZ)\cong \bbR\times S^1\), take \(f\) constant, and let \(u([z])=\Re z\). This is a complete steady K\"ahler-Ricci soliton with a nonconstant pluriharmonic function satisfying \(|\nabla u|\equiv 1\). Its volume growth is linear, and 
	\[
	    \int_M\frac{|\nabla u|^p}{\Psi(d(o,x))}\dd v< \infty 
	\]
	whenever \(\int_0^\infty \frac{\dd t}{\Psi (t)}<\infty \). Thus the divergence in \eqref{eq-z} cannot be replaced by convergence in a theorem covering all complete steady solitons.
\end{remark}

We now turn to the shrinking case. We give an affirmative answer to the problem we posed above in the range \(0<p\le 2\). 

\begin{theorem}\label{thm-b}
	Let \((M,g,J,f)\) be a connected complete gradient shrinking K\"ahler-Ricci soliton and let \(u\) be a pluriharmonic function on \(M\). If 
	\[
	    \int_M|\nabla u|^p\dd v<\infty 
	\]
	for some \(0<p\le 2\), then \(u\) is a constant function.
\end{theorem}

When \(p>2\), we are currently unable to directly provide a positive answer to the problem we posed above. However, we can present a Liouville theorem under a weighted integral condition, which also encompasses Theorem~\ref{thm-b}. Let \(f_0=f-\min\limits_Mf\ge 0\) and define the smooth function
\[
	\rho(x)=\left(1+\frac{2f_0(x)}{\lambda }\right)^{1/2}.
\]

\begin{theorem}\label{thm-c}
	Let \((M,g,J,f)\) be a connected complete gradient shrinking K\"ahler-Ricci soliton and let \(u\) be a pluriharmonic function on \(M\). Suppose that
    \begin{equation}\label{eq-x}
    	\int_M\omega(\rho )|\nabla u|^p\dd v<\infty 
    \end{equation}
    for some \(0<p<\infty \), where \(\omega :[1,\infty )\to (0,\infty )\) is a continuous positive function satisfying 
	\begin{equation}\label{eq-w}
		\int_1^\infty \frac{\sqrt{\omega (t)}}{t^{p/2}}\dd t=\infty .
	\end{equation}
	Then \(u\) is a constant function.
\end{theorem}

\begin{remark}
	The condition \eqref{eq-w} is precisely the condition arising from the construction of the cutoff functions used in the proof of Theorem~\ref{thm-c}. It guarantees the existence of a sequence of cutoff functions \(\{\eta_R\}\) such that
    \[
        \int_M |H|^p|\nabla\eta_R|^2\,dv\to0 .
    \]
\end{remark}

\begin{remark}
    Let \(\omega \equiv 1\). When \(0<p\le 2\), the integral in \eqref{eq-w} becomes \(\int_1^\infty \frac{1}{t^{p/2}}\dd t\), which is automatically divergent, so Theorem~\ref{thm-c} indeed contains Theorem~\ref{thm-b}. However, when \(p>2\), the integral converges, and at present, our method cannot deduce that \(u\) is constant. Therefore, in the shrinking case, we can currently answer the problem we posed above only for \(0<p\le 2\).
\end{remark}

The remainder of this paper is organized as follows. In Section~\ref{sec-2}, we introduce some notation and establish several auxiliary lemmas that will be used in Sections~\ref{sec-3} and \ref{sec-4}. In Sections~\ref{sec-3} and \ref{sec-4}, we prove the Liouville theorems in the steady and shrinking cases, respectively. In Section~\ref{sec-5}, we provide an example showing that the soliton equation is essential for our Liouville theorems to extend to the range \(0<p<1\). 

\section{Preliminaries and auxiliary lemmas}\label{sec-2}

We first fix some notation. We use local holomorphic coordinates \(z=(z^1,\ldots,z^m)\) on a K\"ahler manifold, and Greek indices refer to the \((1,0)\)-components. The Einstein summation convention is understood. All subscripts denote covariant derivatives unless explicitly stated as ordinary partial derivatives.

For a real-valued smooth function \(v\), we write
\[
    v_\alpha=\nabla_\alpha v,\qquad v_{\bar\alpha}=\nabla_{\bar\alpha}v .
\]
For scalar functions, these first-order covariant derivatives agree with the ordinary partial derivatives. The underlying Riemannian metric is expressed in holomorphic coordinates as
\[
    g=2\Re\left(g_{\alpha \bar\beta }\,dz^\alpha \otimes d\bar z^\beta \right).
\]
Accordingly, we have
\[
    \nabla v=g^{\alpha \bar\beta }\left(v_{\bar\beta }\partial_\alpha +v_\alpha\partial_{\bar\beta}\right)=2\Re\left(g^{\alpha \bar\beta }v_{\bar\beta }\partial_\alpha \right),
\]
and
\[
    \Delta v=g^{\alpha \bar\beta }v_{\alpha \bar\beta }+g^{\beta \bar\alpha }v_{\bar\alpha \beta }=2g^{\alpha \bar\beta }v_{\alpha \bar\beta }.
\]
For two real-valued smooth functions \(v,w\),
\[
    \langle \nabla v,\nabla w\rangle =g^{\alpha \bar\beta }v_\alpha w_{\bar\beta }+g^{\alpha \bar\beta }v_{\bar\beta }w_\alpha =2\Re \left(g^{\alpha \bar\beta }v_\alpha w_{\bar\beta }\right).
\]

For later use, we recall the following standard identities for gradient Ricci solitons
\begin{equation}\label{eq-az}
    R+|\nabla f|^2-2\lambda f=C_0,
\end{equation}
where \(C_0\) is a constant (see \cite{MR2648937,MR2274812,MR2448435}). Moreover, the scalar curvature of complete gradient steady and shrinking Ricci solitons is nonnegative, that is, \(R\ge 0\) (see \cite{MR2520796}).

The following Caccioppoli-type estimate is a localized form of Yau’s classical argument for the \(L^p\) Liouville theorem for holomorphic functions \cite{MR417452}. For a related regularized cutoff calculation on complete Gauduchon manifolds, compare \cite{MR3951749}.
\begin{lemma}\label{lem-a}
	Let \(h\) be holomorphic on a K\"ahler manifold \(M\) and let \(0<p<\infty \). Then for every compactly supported locally Lipschitz function \(\eta \), we have
	\begin{equation}\label{eq-a}
		\int_M\eta^2\left|\nabla(|h|^{p/2})\right|^2\dd v\le \int_M|h|^p|\nabla \eta |^2\dd v.
	\end{equation}
	The left-hand side is understood by regularization at the zero set. In particular, if \(M\) is connected, \(\eta_j\to 1\) locally, and the right-hand side tends to zero, then \(h\) is constant.
\end{lemma}

\begin{proof}
	Fix \(0<\varepsilon \le 1\) and set \(Q=|h|^2+\varepsilon \). Since \(h\) is holomorphic,
	\[
	    |\nabla Q|^2=2|h|^2|\partial h|^2.
	\]
	The chain rule gives
    \begin{equation}\label{eq-b}
    	\begin{split}
    		\Delta Q^{p/2}=&\frac{p}{2}Q^{(p-2)/2}\Delta Q+\frac{p}{2}\left(\frac{p}{2}-1\right)Q^{(p-4)/2}|\nabla Q|^2\\
    		=&pQ^{(p-4)/2}\left(\varepsilon +\frac{p}{2}|h|^2\right)|\partial h|^2.
    	\end{split}
    \end{equation}
    On the other hand, 
    \[
        \left|\nabla (Q^{p/4})\right|^2=\frac{p^2}{16}Q^{(p-4)/2}|\nabla Q|^2=\frac{p^2}{8}Q^{(p-4)/2}|h|^2|\partial h|^2.
    \]
    Combining \eqref{eq-b}, this gives 
    \begin{equation}\label{eq-c}
    	\Delta Q^{p/2}\ge 4\left|\nabla (Q^{p/4})\right |^2.
    \end{equation}
    Multiplying \eqref{eq-c} by \(\eta^2\) and integrating by parts, we obtain
    \begin{align*}
    	4\int_M\eta^2\left|\nabla (Q^{p/4})\right|^2\dd v
    	\le &-\int_M\left \langle \nabla (\eta^2),\nabla (Q^{p/2})\right \rangle \dd v\\
    	\le &4\int_M|\eta |Q^{p/4}|\nabla \eta |\left|\nabla (Q^{p/4})\right|\dd v\\
    	\le &4\left(\int_M\eta^2\left|\nabla (Q^{p/4})\right|^2\dd v\right)^{\frac{1}{2}}\left(\int_MQ^{p/2}|\nabla \eta |^2\dd v\right)^{\frac{1}{2}},
    \end{align*}
    where the last inequality comes from the Cauchy--Schwarz inequality. Hence we get
    \begin{equation}\label{eq-d}
    	\int_M\eta^2\left|\nabla (Q^{p/4})\right|^2\dd v\le \int_MQ^{p/2}|\nabla \eta |^2\dd v.
    \end{equation}

    Now choose a relatively compact open set \(U\Subset M\) such that \(\supp \eta \Subset U\), and choose \(\xi \in C^\infty_c(M)\) satisfying
    \[
        0\le \xi \le 1,\qquad \xi |_U\equiv 1.
    \]
    Applying \eqref{eq-d} and replacing \(\eta \) with \(\xi \), we obtain
    \[
        \int_U\left|\nabla (Q^{p/4})\right|^2\dd v\le \int_MQ^{p/2}|\nabla \xi |^2\dd v\le C,
    \]
    where \(C\) is independent of \(\varepsilon \). Moreover,
    \[
        \int_U\left|Q^{p/4}\right|^2\dd v=\int_UQ^{p/2}\dd v\le C.
    \]
    Thus \(Q^{p/4}\) is uniformly bounded in \(W^{1,2}(U)\). Since \(Q^{p/4}\ge |h|^{p/2}\), we have
    \[
        \left|Q^{p/4}-|h|^{p/2}\right|^2\le Q^{p/2}-|h|^p.
    \]
    The right-hand side converges to \(0\) in \(L^1(U)\) by dominated convergence. It follows that \(Q^{p/4}\longrightarrow |h|^{p/2}\) in \(L^2(U)\). By the uniform \(W^{1,2}(U)\) bound, after passing to a subsequence, \(Q^{p/4}\) converges weakly to \(|h|^{p/2}\) in \(W^{1,2}(U)\), where the weak limit is identified using the above strong \(L^2(U)\) convergence. Therefore, weak lower semicontinuity gives
    \begin{equation}\label{eq-e}
    	\int_M\eta^2\left|\nabla \left(|h|^{p/2}\right)\right|^2\dd v
    	\le \liminf_{\varepsilon \downarrow 0}\int_M\eta^2\left|\nabla (Q^{p/4})\right|^2\dd v.
    \end{equation}
    By dominated convergence, the right-hand side of \eqref{eq-d} convergences to \(\int_M|h|^p|\nabla \eta |^2\dd v\). Together with \eqref{eq-e}, yields \eqref{eq-a}.

    Now suppose that \(\eta_j\to 1\) locally uniformly and that the right-hand side of \eqref{eq-a} tends to zero. Fix a compact set \(K\subset M\). For all sufficiently large \(j\), we have \(\eta_j\ge 1/2\) on \(K\). Hence, by \eqref{eq-a},
    \[
        \frac14\int_K\left|\nabla \left(|h|^{p/2}\right)\right|^2\,\dd v 
        \le \int_M\eta_j^2\left|\nabla \left(|h|^{p/2}\right)\right|^2\,\dd v 
        \le \int_M|h|^p|\nabla\eta_j|^2\,\dd v.
    \]
    Letting \(j\to\infty\), we obtain
    \[
        \int_K\left|\nabla \left(|h|^{p/2}\right)\right|^2\,\dd v=0.
    \]
    Since \(K\) is arbitrary, \(|h|^{p/2}\) has zero weak gradient on \(M\). As \(M\) is connected and \(|h|\) is continuous, \(|h|\) is constant. Since \(h\) is holomorphic and has constant modulus, the open mapping theorem implies that \(h\) is constant.
\end{proof}

The following properties follow from the quadratic growth estimate of the potential function due to \cite{MR2732975} or \cite{MR2846384}.
\begin{lemma}\label{lem-b}
	Let \((M,g,f)\) be a complete noncompact gradient shrinking Ricci soliton. The potential is bounded below, proper, and attains its minimum. Moreover, for every fixed minimum point \(o\) of \(f\), there is \(C\ge 1\) such that
	\[
	    C^{-1}(1+d(o,x))\le \rho(x)\le C(1+d(o,x)),\qquad |\nabla \rho |(x)\le C,\qquad |\nabla f|(x)\le C\rho (x).
	\]
	In particular, all sublevel sets of \(f\) and \(\rho \) are compact.
\end{lemma}

\begin{proof}
	Under a scaling of the metric, and by invoking the quadratic growth estimate for the potential due to \cite{MR2732975} or \cite{MR2846384}, after adding a suitable constant to \(f\), there exist a base point \(q\) and \(a,A>0\) and \(b,B\ge 0\) such that 
	\begin{equation}\label{eq-f}
        ad(q,x)^2-b\le f(x)\le Ad(q,x)^2+B.
	\end{equation}
	It can be concluded that when \(a\to \infty \), \(f(x)\to \infty \). Therefore, \(f(x)\) has a lower bound and is a proper function. Since \(M\) is complete, by the Hopf--Rinow theorem, each sub-level set of \(f(x)\) is compact. Therefore, \(f(x)\) attains its minimum value at some point \(o\). By the triangle inequality, 
	\[
	    d(o,x)\le d(q,x)+d(o,q),\qquad d(q,x)\le d(o,x)+d(0,q).
	\]
	It follows that
	\[
	    \frac{1}{2}d(o,x)^2-d(o,q)^2\le d(q,x)^2\le 2d(o,x)^2+2d(o,q)^2
	\]
    Substitute this into \eqref{eq-f} and absorb the constants, we can obtain that there exist \(a_0,A_0>0\) and \(b_0,B_0\ge 0\) such that
    \begin{equation}\label{eq-g}
        a_0d(o,x)^2-b_0\le f_0(x)\le A_0d(o,x)^2+B_0.
    \end{equation}

    From \eqref{eq-g}, it can be concluded that \(\rho(x)\) and \(1+d(o,x)\) are comparable. In fact, we have
    \begin{equation}\label{eq-h}
    	\begin{split}
    		\rho^2
    		=&1+\frac{2f_0}{\lambda }\\
    		\le &1+\frac{2A_0}{\lambda }d(o,x)^2+\frac{2B_0}{\lambda }\\
    		\le &C_1(1+d(o,x))^2.
    	\end{split}
    \end{equation}
    Choose \(R_0\ge \max\left\{1,\sqrt{\frac{2b_0}{a_0}}\right\}\). When \(d(o,x)\ge R_0\), by \eqref{eq-g}, we have \(f_0\ge \frac{a_0}{2}d(o,x)^2\). Hence,
    \begin{equation}\label{eq-i}
    	\rho^2
    	\ge \frac{2f_0}{\lambda }
        \ge \frac{a_0}{\lambda }d(o,x)^2
        \ge \frac{a_0}{4\lambda }(1+d(o,x))^2.
    \end{equation}
    When \(d(0,x)<R_0\), we have 
    \begin{equation}\label{eq-j}
    	\rho^2\ge 1\ge \frac{1}{(1+R_0)^2}(1+d(o,x))^2.
    \end{equation}
    Combining \eqref{eq-h}, \eqref{eq-i} and \eqref{eq-j}, and adjusting the constants, we obtain
    \[
        C^{-1}(1+d(o,x))\le \rho(x)\le C(1+d(o,x)).
    \]

    We now proceed to derive the gradient estimates. After replacing \(f\) by the normalized potential \(f_0=f-f_{\min}\), the identity \eqref{eq-az} remains valid in the form
    \[
        R+|\nabla f|^2-2\lambda f_0=C_0,
    \]
    where we still denote the constant by \(C_0\). However, \(f_0(o)=0\) and \(\nabla f(o)=0\), substituting these values into the above equality yields \(C_0=R(o)\). Consequently, the above equality reduces to
    \begin{equation}\label{eq-k}
    	R+|\nabla f|^2-2\lambda f_0=R(o).
    \end{equation}
    Chen \cite{MR2520796} proved that \(R\ge 0\). By \eqref{eq-k} and the definition of \(\rho \), we have
    \begin{equation}\label{eq-l}
        \begin{split}
    	    |\nabla f|^2
    	    \le &2\lambda f_0+R(o)\\
    	    =&\lambda^2(\rho^2-1)+R(o)\\
    	    \le &\left(\lambda^2+R(o)\right)\rho^2,
    	\end{split}
    \end{equation}
    where the last inequality follows from \(\rho \ge 1\). Therefore, we arrive at
    \[
        |\nabla f|\le \sqrt{\lambda^2+R(o)}\rho \le C\rho .
    \]
    Taking the gradient of \(\rho^2=1+\frac{2f_0}{\lambda }\) yields
    \[
        \nabla \rho =\frac{\nabla f}{\lambda \rho }.
    \]
    combining \eqref{eq-l}, we obtain
    \begin{align*}
    	|\nabla \rho |^2
    	\le &\frac{2\lambda f_0+R(o)}{\lambda^2\rho^2}\\
    	=&\frac{\lambda^2(\rho^2-1)+R(o)}{\lambda^2\rho^2}\\
    	\le &1+\frac{R(o)}{\lambda^2}.
    \end{align*}
    That is, \(|\nabla \rho |\le C\). Finally, as the sublevel sets of \(f\) are known to be compact, we can immediately conclude from the relationship between \(\rho \) and \(f\) that the sublevel sets of \(\rho \) are compact as well. 
\end{proof}

Finally, we give a refined Kato inequality which is responsible for the extension below
\(p=1\) in the steady case. Although this inequality follows from the harmonic \(1\)-form argument in \cite[Theorem 4.2]{MR2657671}, we provide a direct proof for completeness.

\begin{lemma}\label{lem-c}
	Let \(u\) be a real-valued pluriharmonic function on a K\"ahler manifold \(M\). At every point at which \(|\nabla u|>0\),
	\begin{equation}\label{eq-v}
		|\nabla^2u|^2\ge 2|\nabla |\nabla u||^2.
	\end{equation}
	The inequality holds almost everywhere on \(M\) when \(\nabla |\nabla u|\) is interpreted as the weak gradient.
\end{lemma}

\begin{proof}
	Fix \(x\) with \(|\nabla u|(x)>0\) and choose K\"ahler normal holomorphic coordinates. After a unitary change of coordinates and a phase change in \(z^1\), we may arrange
    \[
        u_1(x)=\frac{|\nabla u|(x)}{\sqrt{2}},\qquad u_\beta (x)=0,\ \beta >1.
    \]
    Since \(u_{\bar\alpha \gamma =0}\), differentiation of \(|\nabla u|^2=2u_\alpha u_{\bar\alpha }\) gives
    \[
        (|\nabla u|^2)_\gamma =2u_{\alpha \gamma }u_{\bar\alpha }=\sqrt{2}su_{1\gamma }.
    \]
    It follows that
    \[
        |\nabla |\nabla u||^2=\sum_\gamma |u_{1\gamma }|^2.
    \]
    Therefore, 
    \[
        |\nabla^2u|^2=2\sum_{\alpha ,\beta }|u_{\alpha \beta }|^2\ge 2\sum_\gamma |u_{1\gamma }|^2=2|\nabla |\nabla u||^2.
    \]
    The function \(|\nabla u|^2\) is smooth, while \(|\nabla u|\) is continuous and locally Lipschitz. The weak gradient of a locally Lipschitz function vanishes almost everywhere on a level set, so the almost-everywhere extension follows.
\end{proof}

\section{Liouville theorems for gradient steady K\"ahler–Ricci solitons}\label{sec-3}

In this section, we prove Theorem~\ref{thm-a}. To this end, we need several lemmas. 

\begin{lemma}\label{lem-d}
	Let \((M,g)\) be a complete noncompact Riemannian manifold and suppose \(V(R)=\Vol B_R(o)\ge cR\) for all large \(R\). If \(\Psi \ge 1\) is nondecreasing and \eqref{eq-z} holds, then
	\begin{equation}\label{eq-q}
	    \int_M\frac{\dd v}{\Psi(d(o,x))}=\infty .
	\end{equation}
\end{lemma}

\begin{proof}
	Set \(w=\frac{1}{\Psi }\), so \(w\) is nonincreasing. By the co-area formula and integration by parts,
	\begin{align*}
		\int_{B_R(o)}w(d(o,x))\dd v
		=&\int_0^Rw(t)\dd V(t)\\
		=&w(R)V(R)-\int_0^RV(t)w'(t)\dd t.
	\end{align*}
	Since \(w'\le 0\) and \(V(t)\ge ct\) for \(t\ge R_0\),
	\begin{align*}
		\int_{B_R(o)}w(d(o,x))\dd v
		\ge &cRw(R)-c\int_{R_0}^Rtw'(t)\dd t-C\\
		=&c\int_{R_0}^Rw(t)\dd t-C',
	\end{align*}
	which tends to infinity by \eqref{eq-z}. 
\end{proof}

Define 
\[
	\Theta (t)=\int_0^t\frac{\dd \tau }{\Psi(\tau )}.
\]
Then \(\Theta (t)\to \infty \). Choose a fixed Lipschitz function \(\chi :[0,\infty )\to [0,1]\) with \(\chi =1\) on \([0,1]\), \(\chi =0\) on \([2,\infty )\), and \(|\chi '|\le 2\), and put
\begin{equation}\label{eq-m}
	\eta_k(x)=\chi\left(\frac{\Theta (d(o,x))}{k}\right).
\end{equation}
By the Hopf--Rinow theorem and the divergence of \(\Theta \), these cutoffs have compact support, converge to \(1\) locally uniformly, and satisfy almost everywhere 
\begin{equation}\label{eq-n}
	|\nabla \eta _k|\le \frac{C}{k\Psi(d(o,x))}.
\end{equation}

\begin{lemma}\label{lem-e}
	Under the assumptions of Theorem~\ref{thm-a}, \(H=0\) and hence \(F=0\). 
\end{lemma}

\begin{proof}
	Given that all harmonic function on a compact manifold are constants, only the noncompact case needs to considered. In the steady case, \eqref{eq-az} reduces to \(R+|\nabla f|^2=C_0\), where \(C_0\) is a nonnegative constant. If \(C_0=0\), then \(\nabla f=0\). It follows from \(|H|\le \frac{1}{2}|\nabla f||\nabla u|\) that \(H=0\). Assume \(C_0>0\), we have \(H\le \frac{\sqrt{C_0}}{2}|\nabla u|\). Apply Lemma~\ref{lem-a} with \(h=H\) and \(\eta =\eta_k\). Using \(\Psi \ge 1\),
	\begin{align*}
	    \int_M|H|^p|\nabla \eta_k|^2\dd v
	    \le &\frac{C}{k^2}\int_M\frac{|\nabla u|^p}{\Psi(d(o,x))^2}\dd v\\
	    \le &\frac{C}{k^2}\int_M\frac{|\nabla u|^p}{\Psi(d(o,x))}\dd v\longrightarrow 0.
	\end{align*}
	Thus \(H\) is constant. If it were nonzero, then the bound \(H\le \frac{\sqrt{C_0}}{2}|\nabla u|\) gives \(|\nabla u|\ge \Lambda \) for some positive constant \(\Lambda \), and gradient steady Ricci solitons have at least linear volume growth (see \cite{MR3023848}), which allows us to apply Lemma~\ref{lem-d}. These imply that
	\[
	\int_M\frac{|\nabla u|^p}{\Psi(d(o,x))}\dd v\ge \Lambda^p\int_M\frac{\dd v}{\Psi (d(o,x))}=\infty ,
	\]
	which contradicts \eqref{eq-y}. Therefore, \(H=0\), and consequently \(F=0\).
\end{proof}

\begin{lemma}\label{lem-f}
	Let \((M,g,J,f)\) be a complete gradient steady K\"ahler-Ricci soliton. Let \(u\) be a pluriharmonic function satisfying \(F=\langle \nabla f,\nabla u\rangle =0\). Then for every \(p>0\), every compactly supported locally Lipschitz function \(\eta \), and every \(0<\delta <1\),
	\begin{equation}\label{eq-o}
		(1-\delta )c_p\int_M\eta^2|\nabla u|^{p-2}|\nabla |\nabla u||^2\dd v+\frac{1}{p}\int_MR\eta^2|\nabla u|^p\dd v\le \frac{1}{\delta c_p}\int_M|\nabla u|^p|\nabla \eta |^2\dd v+\frac{1}{p}\int_M|\nabla f||\nabla (\eta^2)||\nabla u|^p\dd v,
	\end{equation}
    where \(c_p=\min\{2,p\}\). For \(p<2\), the first integrand is defined to be zero almost everywhere on \(\{|\nabla u|=0\}\).
\end{lemma}

\begin{proof}
	Throughout this proof, we write \(s=|\nabla u|\). By taking the covariant derivative of \(\langle \nabla f,\nabla u\rangle =0\) with respect to \(\nabla u\) and using the soliton equation in the steady case, we obtain
	\begin{align*}
	    0=&\nabla^2f(\nabla u,\nabla u)+\nabla^2u(\nabla f,\nabla u)\\
	    =&-\Ric (\nabla u,\nabla u)+\frac{1}{2}\langle \nabla f,\nabla (s^2)\rangle .
	\end{align*}
	This, together with the Bochner formula and the Kato inequality for \(\dd u\) in \cite{MR2657671}, yields
    \begin{equation}\label{eq-p}
    	\begin{split}
    		\frac{1}{2}\Delta s^2
    		=&|\nabla^2u|^2+\Ric (\nabla u,\nabla u)\\
            \ge &2|\nabla s|^2+\frac{1}{2}\langle \nabla f,\nabla (s^2)\rangle .
    	\end{split}
    \end{equation}

    Now fix \(\varepsilon >0\), and set \(q=s^2+\varepsilon \) and \(\alpha =\frac{p-2}{2}\). We multiply \eqref{eq-p} by \(\eta^2q^\alpha \) and integrating over \(M\), we obtain
    \begin{equation}\label{eq-s}
    	\frac{1}{2}\int_M\eta^2q^\alpha \Delta (s^2)\dd v
    	\ge 2\int_M\eta^2q^\alpha |\nabla s|^2\dd v
    	+\frac{1}{2}\int_M\eta^2q^\alpha \langle \nabla f,\nabla (s^2)\rangle \dd v.
    \end{equation}
    Integration by parts gives
    \begin{equation}\label{eq-t}
        \begin{split}
        	\frac{1}{2}\int_M\eta^2q^\alpha \Delta (s^2)\dd v
    	    =&-\frac{1}{2}\int_M\langle \nabla(\eta^2q^\alpha ),\nabla (s^2)\rangle \dd v\\
    	    =&-\int_M\eta q^\alpha \langle \nabla \eta ,\nabla (s^2)\rangle \dd v
    	    -\frac{\alpha }{2}\int_M\eta^2q^{\alpha -1}|\nabla (s^2)|^2\dd v\\
    	    =&-2\int_M\eta q^\alpha s\langle \nabla \eta ,\nabla s\rangle \dd v
    	    -(p-2)\int_M\eta^2q^{(p-4)/2}s^2|\nabla s|^2\dd v
        \end{split}
    \end{equation}
    Combining \eqref{eq-s} and \eqref{eq-t} and rearranging terms gives  
    \begin{equation}\label{eq-u}
    	2\int_M\eta^2q^\alpha |\nabla s|^2\dd v
    	+(p-2)\int_M\eta^2q^{(p-4)/2}s^2|\nabla s|^2\dd v
    	\le -\frac{1}{2}\int_M\eta^2q^\alpha \langle \nabla f,\nabla (s^2)\rangle \dd v
    	-2\int_M\eta q^\alpha s\langle \nabla \eta ,\nabla s\rangle \dd v.
    \end{equation}
    Since \(0\le s^2/q\le 1\), we have \(2+(p-2)s^2/q\ge c_p\), and consequently the left-hand side of \eqref{eq-u}
    \begin{equation}\label{eq-aa}
        2\int_M\eta^2q^\alpha |\nabla s|^2\dd v
    	+(p-2)\int_M\eta^2q^{(p-4)/2}s^2|\nabla s|^2\dd v
    	\ge c_p\int_M\eta^2q^{(p-2)/2}|\nabla s|^2\dd v.
    \end{equation}
    For the second term on the right-hand side of \eqref{eq-u}, noting that \(s^2\le q\), Young's inequality gives
    \begin{equation}\label{eq-ab}
    	2\left|\int_M\eta q^\alpha s\langle \nabla \eta ,\nabla s\rangle \dd v\right|
    	\le \delta c_p\int_M\eta^2q^{(p-2)/2}|\nabla s|^2\dd v
    	+\frac{1}{\delta c_p}\int_Mq^{p/2}|\nabla \eta |^2\dd v. 
    \end{equation}
    Now we treat the first term on the right-hand side of \eqref{eq-u}. Since \(q^\alpha \nabla (s^2)=q^\alpha \nabla q=\frac{2}{p}\nabla (q^{p/2})\), by using Integration by parts, we have
    \begin{equation}\label{eq-ac}
        \begin{split}
            -\frac{1}{2}\int_M\eta^2q^\alpha \langle \nabla f,\nabla (s^2)\rangle \dd v
            =&-\frac{1}{p}\int_M\eta^2\langle \nabla f,\nabla (q^{p/2})\rangle \dd v\\
            =&\frac{1}{p}\int_Mq^{p/2}\divg (\eta^2\nabla f)\dd v\\
            =&\frac{1}{p}\int_Mq^{p/2}\langle \nabla (\eta^2),\nabla f\rangle \dd v
            +\frac{1}{p}\int_M\eta^2q^{p/2}\Delta f\dd v\\
            \le &\frac{1}{p}\int_M|\nabla (\eta^2)||\nabla f|q^{p/2}\dd v
            -\frac{1}{p}\int_MR\eta^2q^{p/2}\dd v,
        \end{split}
    \end{equation}
    where for the last equality we used \(\Delta f=-R\), which is the trace of the steady soliton equation. Combining \eqref{eq-u}, \eqref{eq-aa}, \eqref{eq-ab} and \eqref{eq-ac}, we obtain
    \[
        (1-\delta )c_p\int_M\eta^2q^{(p-2)/2}|\nabla s|^2\dd v
        +\frac{1}{p}\int_MR\eta^2q^{p/2}\dd v
        \le \frac{1}{\delta c_p}\int_Mq^{p/2}|\nabla \eta |^2\dd v
        +\frac{1}{p}\int_M|\nabla (\eta^2)||\nabla f|q^{p/2}\dd v.
    \]
    Let \(\varepsilon \downarrow 0\), since \(q^{p/2}\longrightarrow s^p\), and by Fatou's lemma
    \[
        \int_M\eta^2s^{p-2}|\nabla s|^2\dd v\le \liminf_{\varepsilon \downarrow 0}\int_M\eta^2q^{(p-2)/2}|\nabla s|^2\dd v,
    \]
    we obtain \eqref{eq-o}.
\end{proof}

We now prove Theorem~\ref{thm-a}.

\begin{proof}[Proof of Theorem~\ref{thm-a}]
	If \(M\) is compact, harmonicity gives the conclusion. Assume \(M\) is noncompact. By Lemma~\ref{lem-e}, \(F=0\), so Lemma~\ref{lem-f} applies. Use the cutoffs \eqref{eq-m}. From \eqref{eq-n}, together with \(\Psi \ge 1\), and \(|\nabla f|\le C\), we have
    \begin{align*}
    	&\int_M|\nabla u|^p|\nabla \eta_k|^2\dd v\le \frac{C}{k^2}\int_M\frac{|\nabla u|^p}{\Psi(d(o,x))^2}\dd v\le \frac{C}{k^2}\int_M\frac{|\nabla u|^p}{\Psi(d(o,x))}\dd v\longrightarrow 0,\\
    	&\int_M|\nabla f||\nabla (\eta_k^2)||\nabla u|^p\dd v\le \frac{C}{k}\int_M\frac{|\nabla u|^p}{\Psi(d(o,x))}\longrightarrow 0.
    \end{align*}
    Since \(R\ge 0\), Fatou's lemma in \eqref{eq-o} gives
    \begin{equation}\label{eq-r}
    	\int_M|\nabla u|^{p-2}|\nabla |\nabla u||^2\dd v=0.
    \end{equation}
    On each connected component of the open set \(\{|\nabla u|>0\}\), the function \(|\nabla u|\) is smooth and \eqref{eq-r} implies that it is a positive constant. By continuity, such a component is also closed. Since \(M\) is connected, either \(|\nabla u|\equiv 0\) or \(|\nabla u|\equiv c>0\) on \(M\). The latter alternative contradicts \eqref{eq-y} and \eqref{eq-q}. Hence \(|\nabla u|=0\), so \(u\) is constant.
\end{proof}

\section{Liouville theorems for gradient shrinking K\"ahler–Ricci solitons}\label{sec-4}

In this section, we prove the Liouville theorems in the shrinking case. Since Theorem~\ref{thm-c} contains Theorem~\ref{thm-b} as a special case, we mainly focus on the proof of Theorem~\ref{thm-c}. 

The following lemma extracts the final step in the proof of \cite[Theorem~0.2]{MR3326575}. There, one first proves that \(u\) is harmonic, comparing \(\Delta_\varphi u=0\) with \(\Delta u=0\) then gives \(\langle\nabla\varphi,\nabla u\rangle=0\). Properness of \(\varphi\) yields constancy by Green’s identity on compact sublevel sets. In our setting, \(H=0\) gives the required orthogonality directly.

\begin{lemma}\label{lem-g}
	Let \((M,g)\) be a connected complete Riemannian manifold, and let \(f,u\in C^\infty(M)\). Assume that \(f\) is proper and bounded below, that \(u\) is harmonic, and that
    \[
        \langle \nabla f,\nabla u\rangle=0
    \]
    everywhere on \(M\). Then \(u\) is constant.
\end{lemma}

\begin{proof}
	If \(M\) is compact, every harmonic function on \(M\) is constant, so assume that \(M\) is noncompact. Since a proper function on a noncompact manifold is unbounded above, Sard's theorem supplies regular values tending to infinity. Choose regular values \(T_j\to \infty \) and set \(\Omega_j=\{x\in M: f(x)<T_j\}\). The properness of \(f\) makes \(\overline{\Omega_j}\) compact, and regularity of \(T_j\) gives a smooth boundary with outward unit normal \(\nu =\frac{\nabla f}{|\nabla f|}\). By \(\langle \nabla f,\nabla u\rangle=0\), \(\partial_\nu u=0\) on \(\partial \Omega_j\). The divergence theorem gives
	\[
	    \int_{\Omega_j}|\nabla u|^2\dd v=\int_{\partial }u\partial_\nu u\dd v_{\partial \Omega_j}-\int_{\Omega_j}u\Delta u\dd v=0.
	\]
	The regular sublevel sets exhaust \(M\), so \(\nabla u=0\) everywhere. The connectedness of \(M\) gives the conclusion.
\end{proof}

To prove Theorem~\ref{thm-c}, we need the following proposition, which gives a Liouville criterion in terms of a suitable sequence of cutoff functions.

\begin{proposition}\label{prop-a}
	Let \((M,g,J,f)\) be a connected complete noncompact gradient shrinking K\"ahler-Ricci soliton, let \(u\) be a real-valued pluriharmonic function on \(M\), and let \(0<p<\infty \). Suppose that there are compactly supported locally Lipschitz function \(0\le \eta_j\le 1\), converging to \(1\) locally uniformly, such that
    \begin{equation}\label{eq-ad}
    	\lim_{j\to \infty }\int_M|H|^p|\nabla \eta_j|^2\dd v=0.
    \end{equation}
    Then \(u\) is constant. 
\end{proposition}

\begin{proof}
	Applying Lemma~\ref{lem-a} to \(h=H\) and \(\eta_j\), then using \eqref{eq-ad}, shows that \(H\) is constant. At a minimum point \(o\) of \(f\), we have \(\nabla f(o)=0\), and hence \(H(o)=0\). Since \(H\) is constant, it follows that \(H=0\), and therefore \(F=2\Re H=0\). Lemma~\ref{lem-g} gives that \(u\) is constant. 
\end{proof}

We can now present the proof of Theorem~\ref{thm-c}.

\begin{proof}[Proof of Theorem~\ref{thm-c}]
    If \(M\) is compact, the conclusion follows immediately from the harmonicity of \(u\). Assume that \(M\) is noncompact. Define
    \[
        A_\omega (t)=\int_1^t\frac{\sqrt{\omega(\tau )}}{\tau^{p/2}}\dd \tau ,\qquad t\ge 1.
    \]
    Since \(\omega \) is positive and continuous, \(A_\omega \) is continuously differentiable and strictly increasing, with \(A_\omega (1)=0\). By the condition \eqref{eq-w}, \(\lim\limits_{t\to \infty }A_\omega (t)=\infty \). For \(R>1\), define
    \[
        \eta_R(x)=\left(1-\frac{A_\omega (\rho(x))}{A_\omega (R)}\right)^+,
    \]
    where \(a^+:=\max\{a,0\}\) denotes the positive part of \(a\). It follows that \(0\le \eta_R\le 1\) and \(\supp \eta_R\subset \{\rho \le R\}\). The latter set is compact because \(\rho \) is proper by Lemma~\ref{lem-b}. For every compact \(K\subset M\), the function \(\rho \) is bounded on \(K\), and hence \(\lim\limits_{R\to \infty }\sup\limits_K\frac{A_\omega (\rho )}{A_\omega (R)}=0\). Therefore, \(\eta_R\) converges \(1\) locally uniformly on \(M\). Using \(|\nabla \rho |\le C\) (Lemma~\ref{lem-b}), we have 
    \[
        |\nabla \eta_R|\le \frac{C}{A_\omega (R)}\frac{\sqrt{\omega (\rho )}}{\rho^{p/2}}
    \]
    almost everywhere on \(\{\rho <R\}\). Again by Lemma~\ref{lem-b},
    \[
        |H|\le \frac{1}{2}|\nabla f||\nabla u|\le C\rho |\nabla u|.
    \]
    It follows that
    \begin{align*}
    	\int_M|H|^p|\nabla \eta_R|^2\dd v
    	\le &\frac{C}{A_\omega (R)^2}\int_{\{\rho <R\}}\rho^p|\nabla u|^p\frac{\omega(\rho )}{\rho^p}\dd v\\
    	\le &\frac{C}{A_\omega (R)^2}\int_M\omega(\rho )|\nabla u|^p\dd v.
    \end{align*}
    The last integral is finite by assumption \eqref{eq-x}. Consequently, 
    \[
        \lim_{R\to \infty }\int_M|H|^p|\nabla \eta_R|^2\dd v=0.
    \]
    Choose any sequence \(R_j\to \infty \) and set \(\eta_j=\eta_{R_j}\). Then \(\{\eta_j\}\) satisfies the assumption of Proposition~\ref{prop-a}. Hence \(u\) is constant.
\end{proof}

\section{Necessity of the soliton structure for Liouville theorems in the range \(0<p<1\)}\label{sec-5}

It is essential to distinguish between real harmonic and holomorphic Liouville theorems. Yau's theorem for holomorphic functions holds for \(p>0\) (see \cite{MR417452,MR3951749}), whereas the real harmonic theorem requires \(p>1\). This threshold \(p>1\) is strict on arbitrary complete manifolds, even for pluriharmonic functions in a K\"ahler setting. The following explicit model shows that a Liouville conclusion under the assumption \(|\nabla u|\in L^p(M)\) fails for pluriharmonic functions on general complete K\"ahler manifolds when \(0<p<1\).

\begin{proposition}\label{peop-b}
	Let \(M=\bbR \times S^1\) with coordinates \((t,\theta )\), where \(\theta \) has period \(2\pi \), and let \(g=\dd t^2+e^{-2\cosh t}\dd \theta^2\). Define 
    \begin{equation}\label{eq-ae}
        u(t,\theta )=\int_0^te^{\cosh \tau }\dd \tau .
    \end{equation}
    Then \((M,g)\) is a complete K\"ahler manifold of complex dimension one and \(u\) is a nonconstant real-valued pluriharmonic function. Moreover,
    \begin{equation}\label{eq-af}
    	u\in L^p(M)\qquad \mbox{for every } 0<p\le 1,
    \end{equation}
    and
    \begin{equation}\label{eq-ag}
    	|\nabla u|\in L^p(M)\qquad \mbox{for every } 0<p<1,
    \end{equation}
    while \(|\nabla u|\notin L^1(M)\).
\end{proposition}

\begin{proof}
	Since every oriented surface carries a compatible complex structure, \(M\) becomes a Riemann surface. Moreover, the associated K\"ahler form is a top-degree form and hence closed, so \((M,g)\) is K\"ahler. We now prove it is complete. Set \(a(t)=e^{-\cosh t}\), then 
	\[
	    g=\dd t^2+a(t)^2\dd \theta^2, \qquad \dd v_g=a(t)\dd t\dd \theta .
	\]
	Let \(\gamma (s)=(t(s),\theta(s)):[a,b]\to M\) be a piecewise smooth curve. Then 
	\[
	    |\dot \gamma|_g^2=\dot t^2+a(t)^2\dot \theta ^2\ge \dot t^2.
	\]
	Hence
	\[
	    L_g(\gamma )\ge \int_a^b|\dot t|\dd s\ge |t(\gamma(b))-t(\gamma(a))|.
	\]
	Therefore, \(|t(x)-t(y)|\le d_g(x,y)\) for any \(x,y\in M\). Let \(\{x_j\}\) be a \(d_g\)-Cauchy sequence. Then 
	\[
	    |t(x_i)-t(x_j)|\le d_g(x_i,x_j),
	\]
	so \(\{t(x_j)\}\) converges in \(\bbR\). Therefore, there exists \(R>0\) such that \(\{x_j\}\subset [-R,R]\times S^1\). On this compact cylinder, 
	\[
	    0<c\le a(t)\le C<\infty .
	\]
	Therefore \(g\) is uniformly equivalent to the product metric \(\dd t^2+\dd \theta^2\). Hence the sequence converges in \(M\). Thus \((M,g)\) is complete. The pluriharmonicity of \(u\) follows easily. Since \(u\) depends only on \(t\), we have
	\[
	    \Delta_gu=\frac{1}{a(t)}\frac{\dd }{\dd t}\left(a(t)u'(t)\right)
	\]
	Now \(u'(t)=e^{\cosh t}=1/a(t)\). Consequently, \(a(t)u'(t)=1\), and hence \(\Delta_gu=0\). As \(\dim_{\bbC}M=1\), pluriharmonicity and harmonicity coincide, and hence \(u\) is pluriharmonic.

	We now prove the second part. Since \(e^{\cosh t}\) is even, the function \(u(t)=\int_0^t e^{\cosh\tau}\dd \tau \) is odd. It therefore suffices to study the behavior as \(t\to+\infty\). For \(t>0\), set 
	\[
	    U(t)=\int_0^te^{\cosh \tau }\dd \tau .
	\]
	In the remainder of this paragraph, all asymptotic relations are understood as \(t\to+\infty\). Since 
	\[
        \frac{\dd}{\dd t}\left(\frac{e^{\cosh t}}{\sinh t}\right)=e^{\cosh t}\left(1-\frac{\cosh t}{\sinh^2t}\right)\sim e^{\cosh t},
	\]
	l'H\^opital's rule gives
	\[
	    U(t)\sim \frac{e^{\cosh t}}{\sinh t}.
	\]
	It follows from \(u(t)=U(t)\) that
    \[
        |u(t)|^p\,dv_g\sim \frac{e^{-(1-p)\cosh t}}{(\sinh t)^p}\dd t\dd \theta.
    \]
    If \(0<p<1\), the right-hand side is integrable at \(+\infty\), because \(e^{-(1-p)\cosh t}\) decays super-exponentially. If \(p=1\), then
    \[
        |u(t)|\,dv_g\sim \frac{1}{\sinh t}\dd t\dd \theta,
    \]
    which is integrable at \(+\infty\). Thus \(u\in L^p(M)\) for every \(0<p\le 1\). Finally, because \(u\) depends only on \(t\), \(|\nabla u|=|u'|=e^{\cosh t}\). Therefore,
    \[
        \int_M|\nabla u|^p\dd v_g=2\pi \int_{\bbR}e^{-(1-p)\cosh t}\dd t.
    \]
    This is finite exactly when \(0<p<1\). For \(p=1\), 
    \[
        \int_M|\nabla u|\dd v_g=2\pi \int_{\bbR}\dd t=\infty .
    \]
    The completes the proof.
\end{proof}

\begin{remark}
	There is no conflict between Proposition~\ref{peop-b} and Yau's theorem for holomorphic function. In the oriented orthonormal coframe \(\dd t, e^{-\cosh t}\dd \theta \), one has
	\[
	    *\dd u=*\left(e^{\cosh t}\dd t\right)=\dd \theta .
	\] 
	The form \(\dd \theta \) has period \(2\pi \) around circle and is not exact. Thus \(u\) has local harmonic conjugates but no single-valued global harmonic conjugate. Equivalently, \(u\) is pluriharmonic but is not the real part of a globally defined holomorphic function. The example shows that the soliton structure is essential for the Liouville theorem under the assumption \(|\nabla u|\in L^p(M)\) with \(0<p<1\). Completeness and the K\"ahler condition alone do not suffice.
\end{remark}


\end{document}